\documentclass[12pt]{article}
\usepackage{amsfonts}
\usepackage{amsmath,color}
\usepackage{amssymb}
\usepackage{url}
\usepackage{amsthm}
\hoffset= - 0.9 cm
\voffset= - 1.5 cm
\hoffset= - 0.9 cm
\newtheorem{theorem}{Theorem}

\newtheorem{definition}{Definition}
\newtheorem{example}[theorem]{Example}

\newtheorem{lemma}[theorem]{Lemma}

\newtheorem{proposition}[theorem]{Proposition}
\newtheorem{remark}[theorem]{Remark}

\def\P{\mathbb P}
\def\E{\mathbb E}

\def\R{\mathbb R}

\begin{document}

\title{Strong uniqueness for stochastic evolution equations with
   unbounded measurable drift term
   }

 \author{G. Da Prato \footnote{E-mail: g.daprato@sns.it}\\
{\small Scuola Normale Superiore}
  \\ {\small Piazza dei Cavalieri 7, 56126 Pisa,
 Italy }
  \\ F. Flandoli
  \footnote{E-mail: flandoli@dma.unipi.it}
  \\ {\small  Dipartimento di Matematica Applicata
``U. Dini''} \\ {\small  Universit\`a di Pisa, Italy }\\
  E. Priola
  \footnote{
E-mail:  enrico.priola@unito.it}
 \\ {\small Dipartimento di Matematica ``G. Peano'', Universit\`a di Torino}
 \\ {\small via Carlo Alberto 10,   Torino,  Italy} \\
 M. R\"{o}ckner
 \footnote{Research supported by the DFG through IRTG 1132 and CRC 701
  and the I. Newton Institute, Cambridge, UK.
  E-mail: roeckner@mathematik.uni-bielefeld.de }
\\
{\small Faculty of Mathematics, Bielefeld University}
\\{\small 33501 Bielefeld, Germany}}

\maketitle

\noindent \textbf{Abstract}
 We consider stochastic evolution equations in Hilbert spaces with merely
measurable and locally bounded drift term $B$ and cylindrical Wiener noise.
 We prove pathwise (hence strong) uniqueness in the class of  global  solutions.
 This paper   extends our previous paper
 \cite{DFPR} which generalized Veretennikov's fundamental result  to infinite
dimensions assuming boundedness of the drift term.
  As in \cite{DFPR}  pathwise uniqueness holds for a large class, but not for every
initial condition.
 We also include an application of our result to prove
  existence of strong solutions when the drift   $B$
is only measurable, locally bounded and   grows more than linearly.

{\vskip 4mm }

{\vskip 3mm }

\section{Introduction}

We consider the following abstract stochastic differential equation
 in a
separable Hilbert space $H$
\begin{equation} \label{uno}
dX_{t}=(AX_{t}+B(X_{t}))dt+dW_{t},\;\; t \ge 0,\qquad X_{0}=x\in H,
\end{equation}
where $A:D(A)\subset H\rightarrow H$ is self-adjoint, negative
definite and such that $(-A)^{-1 + \delta}$, for some $\delta \in
(0,1)$,
  is
of trace class,
  $B:H\rightarrow H$ and $W=(W_{t})$ is a cylindrical Wiener process.
In \cite{DFPR} under the assumption that $B$ is Borel measurable and (globally) bounded we prove pathwise uniqueness of solutions to \eqref{uno}. A natural generalization is to extend it to the case where  we only assume that $B$ is \textit{Borel measurable and locally bounded} { (i.e., bounded on balls):}
\begin{equation} \label{bloc}
B\in {B_{b, loc}}(H,H).
\end{equation}
In this paper we prove that assuming \eqref{bloc}  pathwise uniqueness
holds  for $\mu$-a.e. initial condition
 $x$  in the class of  {  global mild solutions} to \eqref{uno}.
   Here $\mu $ denotes the Gaussian measure which is  invariant
 for the Ornstein-Uhlenbeck process $Z= (Z_t)$ which solves \eqref{uno} when
 $B=0$ (see Section 1.1 for more details).

In other words if  for some initial condition $x \in H$, $\mu$-a.e., there exists  a  solution for \eqref{uno} on some filtered probability space $\left(
\Omega, {\mathcal F},
 ({\mathcal F}_{t}), \P \right)$ with
  a
cylindrical $({\mathcal F}_{t})$-Wiener process $W$
then our main result shows that this solution is pathwise  unique.
 This is in  particular the case when
\begin{align} \label{atmost}
B \; \text{is measurable and at most of linear growth}
\end{align}
(i.e., $B$ is measurable and there exists $a,b \ge 0$ such that $|B(x)| \le a + b|x|$, $x \in H$), because then
 existence of weak mild solutions  is
 well-known (see Chapter 10 in \cite{DZ}, \cite{Fe, GG}   and also Appendix in \cite{DFPR}).  Moreover, under condition \eqref{atmost},
 the unique law of any mild solution $X^x$ is equivalent to the law of the Ornstein-Uhlenbeck process starting at $x$  (corresponding to $B=0$).

By our main result, using   a generalization of  the Yamada-Watanabe theorem (see \cite{LR,O}), one deduces that under \eqref{atmost} equation \eqref{uno}
has a unique {  strong mild solution,}  for $\mu$-a.e.
 $x \in H$,  generalizing A. Veretennikov's seminal result \cite{Ver}
 in the case $H = \R^d$ (see also \cite{Fedrizzi, FF, Gyongy-Krylov,
 GyongyMartinez, KR05, tre, Za, Zv74}) to infinite dimensions.

 In Section 4 we  generalize this by relaxing assumption \eqref{atmost}. We
     prove existence of   strong mild solutions,
  starting from  $\mu$-a.e.   initial condition $x \in H$,
   when $B \in {B_{b, loc}}(H,H)$ and moreover
  there
exist $C>0$, $p>0$, such that
\begin{equation} \label{diss}
\left\langle B\left(  y+z\right)  ,y\right\rangle \leq C\big(\left |
y\right |^{2}+ e^{p | z |}+1\big)
\end{equation}
for all $y,z\in H$ (see also Remark \ref{ci3}). Finally in Section 4.3 we show a possible extension of our result by considering local mild solutions.

In order to prove  \textit{pathwise uniqueness} for \eqref{uno} we will consider
bounded truncated drifts like
 \begin{align}\label{bou}
    B_N = B\, 1_{B(0,N)},\;\;\; N \ge 1,
\end{align}
where $B(0,N)$ is the open ball of center 0 and of radius $N$ and $1_{B(0,N)}$ is the indicator function of $B(0,N)$,
  by performing a suitable stopping time argument.
This argument is not straightforward since it must be also used in combination with the Ito-Tanaka trick from \cite{DFPR} (see also \cite{DF, Fla, FGP,   P, Ver}).

In addition in this paper we also simplify some arguments used in \cite{DFPR} in the case of $B \in B_b(H,H)$  (see, in particular, Lemma \ref{new11}).
Before stating our main result precisely, let us recall the following definition
 (cf. \cite{O} and \cite{LR}).
\begin{definition}
 Let $x \in H$.

\noindent (a)  We call  { \it weak mild solution}  to
(\ref{uno}) a tuple   $\left(
\Omega, {\mathcal F},
 ({\mathcal F}_{t}), \P, W, X\right) $, where $\left(
\Omega, {\mathcal F},
 ({\mathcal F}_{t}), \P \right)$ is a  filtered probability space
  on which it is defined a
 cylindrical $({\mathcal F}_{t})$-Wiener process $W$ and
 a continuous $({\mathcal F}_{t})$-adapted $H$-valued
process $X = (X_t) = (X_t)_{t \ge 0}$ such that, $\P$-a.s.,
\begin{equation}
X_{t}=e^{tA}x+\int_{0}^{t}e^{\left(  t-s\right)  A}B\left(
X_{s}\right)
ds+\int_{0}^{t}e^{\left(  t-s\right)  A}dW_{s},\;\;\; t \ge 0.\label{mild}
\end{equation}
(b) A weak mild solution $X$ which is $(\bar {\cal F}_t^W)$-adapted (here $(\bar{\cal F}_t^W)$ denotes the  completed  natural filtration of the cylindrical process $W$)
is called \textit{strong mild solution}.
\end{definition}
 We will
   often use  stopping times
\begin{equation}\label{stop1}
    \tau_N^X = \inf \{t \ge 0 \; :\; X_t \not \in B(0,N) \}
\end{equation}
($\tau_{N}^{X}= + \infty$ if the set is empty), $N \ge 1$.

 In our
main result we consider two mild solutions $X$  and $Y$, having  the same initial condition
$x \in H$, and solving  the same equation \eqref{uno} but   with possibly different drift terms,  respectively $B$ and $B' \in B_{b, loc}(H,H)$, i.e.,
\begin{equation} \label{ci71}
dX_{t}=(AX_{t}+B(X_{t}))dt+dW_{t},\;\; \;\;  X_0=x,
\end{equation}
\begin{equation} \label{ci72}
  dY_{t}=(AY_{t}+B'(Y_{t}))dt+dW_{t},\;\;\;\; Y_0=x.
\end{equation}
\begin{theorem}
\label{main theorem} Assume Hypothesis 1 (see Section 1.1) and let $\mu$ be the centered Gaussian measure on $H$ with covariance $Q = -\frac{1}{2}A^{-1}$.

Then for $\mu$-a.e.
 $x\in H$, if  $X$ and $Y$
  are two weak mild solutions, respectively of \eqref{ci71} and \eqref{ci72},
defined on the same filtered probability
  space $\left(
\Omega, {\mathcal F},
 ({\mathcal F}_{t}), \P \right)$  with the same cylindrical Wiener process $W$,
 and    if,  for some $N \ge 1$,
\begin{equation}\label{d9}
    B(x) = B'(x), \;\;\; x \in B(0,N),
\end{equation}
then,  $\P$-a.s.,
 \begin{equation} \label{rgg}
X_{t\wedge \tau_N^X \wedge \tau_N^Y} = Y_{t\wedge \wedge \tau_N^X \wedge \tau_N^Y },\;\; t \ge 0,
\end{equation}
and so $\tau_N^X = \tau_N^Y$, $\P$-a.s..
\end{theorem}
Above we restrict to $W$ which are cylindrical with respect to the eigenbasis of $A$ (see Section 1.1 for details).
 Clearly if $B=B'$ the result implies that, $\P$-a.s.,
\begin{equation} \label{ci88}
X_t = Y_t,\;\; t \ge 0.
\end{equation}
Indeed  using that $\tau_N^X \uparrow + \infty$
 and $\tau_N^Y \uparrow + \infty$ as $N \to \infty$ (because $X$ and $Y$ are both global solutions) we deduce easily \eqref{ci88} from \eqref{rgg}.

 The proof of Theorem \ref{main theorem}, performed in Section \ref{section proof}, uses a truncation argument and regularity
results for elliptic equations in Hilbert spaces involving truncated drift terms $B_N$ (cf. \eqref{bou}). Such regularity results are  given in Section 2,
where we  also establish an It\^o type formula involving $u(X_t)$
with $u$  in some Sobolev space associated to $\mu$ (see Theorem \ref{p4}).
 In comparison
with  \cite{DFPR}
 to prove such an It\^o type formula  we use a new analytic lemma (see Lemma \ref{new11}).

There are several other quite essential differences in comparison with \cite{DFPR}  in our proof. We refer to Remarks \ref{mi1} and \ref{mi2} for details.

\subsection{Assumptions and preliminaries}

As in \cite{DFPR} we  consider a real separable Hilbert space $H$  and  denote its
norm and inner product  by $\left\vert \, \cdot \, \right\vert $ and
$\left\langle \cdot \, , \cdot \right\rangle $ respectively. Following
\cite{D}, \cite{DZ}, \cite{DZ1}, \cite{PrevRoeckner}, we assume

\vskip 3mm

\noindent {\bf Hypothesis  1}
$A:D(A)\subset H\to H$ is a negative
 definite
self-adjoint operator and
 $(-A)^{-1 + \delta}$, for some $\delta \in (0,1)$,
is of trace class.

\begin{remark} {\em
Our uniqueness result continues to hold under the following  more general assumption:
$A:D(A)\subset H\to H$ is self-adjoint and
there exists $\omega \in \R$ such that
 $(A - \omega)$ is negative definite and
  $(\omega -A)^{-1 + \delta}$, for some $\delta \in (0,1)$,
is of trace class.
 Indeed if we  write  equation \eqref{uno} in the form
$$
dX_{t}=(AX_{t} - \omega X_t) dt \,   + \, (\omega X_t + B(X_{t}))dt+dW_{t},\qquad X_{0}=x\in H,
$$
then  the linear operator $(A - \omega I)$  verifies Hypothesis 1 and the drift  $\omega I + B$ continues to satisfy \eqref{bloc}.
}
\end{remark}

Since $A^{-1}$ is compact, there exists an orthonormal basis $(e_k)$
in $H$ and a sequence of positive numbers $(\lambda_k)$ such that
\begin{equation}
\label{e1a}
Ae_k=-\lambda_k e_k,\quad k\in\mathbb N.
\end{equation}
  Recall that $A$ generates an analytic semigroup $e^{tA}$ on $H$
such that $e^{tA} e_k = e^{- \lambda_k t } e_k$.

 From now on until and including Section 3 we fix $(\Omega,{\cal F}, ({\cal F}_t), \P)$, $W$, $X$ and $Y$ as in the assertion of Theorem \ref{main theorem}.
 As said before we will consider a
cylindrical Wiener process $(W_t)$ with respect to the previous basis
 $(e_k)$. The process $(W_t)$ is formally given by ``$W_t = \sum_{k \ge 1}
  \beta_k(t) e_k$'' where $(\beta_k(t))$ are independent
  one-dimensional Wiener processes (see \cite{DZ} for more details).

By $R_t$ we denote the Ornstein-Uhlenbeck semigroup in $B_b(H)$ (the
Banach space of all Borel and bounded real functions endowed with the
essential supremum norm $\| \, \cdot\, \|_0$)   defined as
\begin{equation}
\label{e1}
R_t f (x)=\int_H f(y)\, N(e^{tA}x,Q_t)(dy),\quad f \in B_b(H), \; x \in H,
\end{equation}
where $N(e^{tA}x,Q_t)$ denotes the Gaussian measure in $H$ of mean
$e^{tA}x$ and covariance operator $Q_t$ given by
\begin{equation}
\label{e2}
Q_t \, = \,  -\frac{1}{2}\;A^{-1}(I-e^{2tA}), \; \quad t\ge 0.
\end{equation}
We  note that $R_t$ has the  unique invariant measure $\mu:=N(0,Q)$
where $Q=-\frac12\;A^{-1}$. Moreover, since under the previous hypotheses, the Ornstein-Uhlenbeck semigroup is strong Feller and
irreducible we have by Doob's theorem that, for any $t>0$, $x \in
H$, the measures $N(e^{tA}x,Q_t)$ and $ \mu$ are equivalent; see \cite{DZ1}. On the other hand, our assumption that $(-A)^{-1+
\delta}$ is trace class guarantees that the OU process
\begin{equation} \label{ou11}
Z_t = Z(t,x)= e^{tA} x + \int_0^t e^{(t-s)A} dW_s
\end{equation}
admits a continuous $H$-valued version.

\smallskip
  If $H$ and $M$ are real separable Hilbert spaces,  the Banach space
  $L^p(H, \mu;$ $ M)$, $p \ge 1$, is defined to consist of equivalent
classes of measurable functions $\varphi: H \to M$ such that $\int_H
|\varphi(x)|_M^p \, \mu(dx) < +\infty $ (if $M = \R$ we set $L^p(H, \mu; \R)=
L^p(H, \mu)$). We will also use the notation $L^p(\mu)$ instead of $L^p(H, \mu, M) $ when no confusion may arise.

The semigroup $R_t$ can be uniquely extended to  a strongly
continuous semigroup of contractions   on $L^p(H,\mu)$, $p \ge 1$,
which we still denote by $R_t$, whereas we  denote by $L_p$ (or $L$
when no confusion may arise) its infinitesimal generator, which is
 defined   on smooth functions $f$ as
$$
L f(x) = \frac{1}{2} \mbox{Tr}(D^2 f(x)) +
 \langle Ax , D f (x)\rangle,
$$
where $D f(x)$ and $D^2 f(x)$ denote, respectively, the
first and second Fr\'echet derivatives of $f$ at $x \in H$.
For real Banach spaces $F$ and $G$ we denote by $C^{k}_b (F,G)$, $k \ge
1$, the Banach space of all functions $f: F \to G$ which are bounded and Fr\'echet differentiable on $F$ up to  order $k \ge 1$ with
all derivatives bounded and continuous on $H$. We also set $C^{k}_b (F, \R)
= C^{k}_b (F)$.

Following  \cite{DZ1}, for any $f \in B_b(H)$ and any $t>0$,
one has
 $R_t f \in C^{\infty}_b(H)  =  \cap _{k \ge 1} C^k_b(H)$. Moreover,
\begin{equation}
\label{e3}
\langle DR_t f (x),k
 \rangle=\int_H \langle
  \Lambda_t k,Q_t^{-\frac12} y\rangle f(e^{tA}x+y)N(0,Q_t)(dy),
  \;\; k \in H, \; x \in H,
\end{equation}
 where $Q_t$ is as defined  in \eqref{e2},
\begin{equation}
\label{e4} \Lambda_t=Q_t^{-1/2}e^{tA}=\sqrt 2\;
(-A)^{1/2}e^{tA}(I-e^{2tA})^{-1/2}
\end{equation}
  and $y \mapsto \langle
  \Lambda_t k,Q_t^{-\frac12} y\rangle$ is a centered
   Gaussian random variable
 under $\mu_t = N(0, Q_t)$ with variance $ |\Lambda_t k|^2$,
 for any $t>0$ (cf. Theorem 6.2.2 in \cite{DZ}).   Since
$$
\Lambda_t e_k =\sqrt
2\;(\lambda_k)^{1/2}e^{-t\lambda_k}(1-e^{-2t\lambda_k})^{-1/2} e_k, \;\;\; k \ge 1,
$$
we see that  there exists
$C_0'>0$ such that
$
\|\Lambda_t \|_{}\le C_0' \,
t^{-\frac12}.
$

In the sequel $\| \cdot \|$  denotes the {\sl Hilbert-Schmidt
norm}; on the other hand $\| \cdot\|_{\mathcal L}$ indicates the
{\sl operator norm.}
By \eqref{e3} we deduce, for any $\varphi \in B_b(H),$
\begin{equation}
\label{e6} \sup_{x \in H}
| DR_t\varphi(x)|
 = \| DR_t\varphi \|_0 \le
C_0 \, t^{-\frac12}\|\varphi\|_0, \;\; t>0,
\end{equation}
which by taking the Laplace transform yields
$$ \| D \, (\lambda - L_2)^{-1}
\varphi\|_{0}\le
\frac{C_{0}}{\lambda^{\frac12}}\;\|\varphi\|_{0},\;\;\; \lambda>0.
$$
Similarly, we find
$$\| DR_t\varphi\|_{L^2(\mu)}\le C_0
t^{-\frac12}\;\|\varphi\|_{L^2(\mu)}
$$
and
$$
\| D \, (\lambda - L_2)^{-1}\varphi
\|_{L^2(\mu)}\le \frac{C_{0}}
{\lambda^{\frac12}}\;\|\varphi\|_{L^2(\mu)}.
$$
  Recall that
  the Sobolev space $W^{2,p}(H, \mu)$, $p \ge 1$, is defined in
 Section 3 of \cite{CG1}
  as the completion of a suitable set of smooth functions
  endowed with  the  Sobolev norm
  (cf. also Section 9.2 in \cite{DZ}
  for the case $p=2$ and \cite{S}).
   Under our above hypotheses,
  the following result is proved in
  Section 10.2.1 of \cite{DZ1}.
 \begin{theorem}
\label{t1}
Let $\lambda>0$,    $f\in L^2(H,\mu)$ and let $v \in D(L_2)$ be the solution of the equation
$$
\lambda v -L_2 v=f.
$$
Then $v\in W^{2,2}(H,\mu)$, $(-A)^{1/2}Dv\in
L^{2}(H,\mu;H)$ and, moreover, there exists  a constant $C(\lambda)$
 such that
$$
 \|v\|_{L^{2}(\mu)} +  \Big ( \int_H \|D^2 v
(x)\|_{}^2 \, \mu (dx)\Big)^{1/2} + \,
\|(-A)^{1/2}Dv\|_{L^{2}(\mu)}\le C\, \|f\|_{L^2(\mu)}.
$$
 \end{theorem}
The following  extension  to $L^p(\mu)$, $p>1$, can be found in
Section 3 of \cite{CG1} (see also \cite{CG} and \cite{MV}).
 \begin{theorem}
\label{t1cg}
Let $\lambda>0$,    $f\in L^p(H,\mu)$ and let $v\in D(L_p)$ be the solution of the equation
$$
\lambda v-L_p  v=f.
$$
Then $v\in W^{2,p}(H,\mu)$, $(-A)^{1/2}Dv\in
L^{p}(H,\mu;H)$ and there exists  a constant $C= C(\lambda, p)$ such
that
$$
\|v\|_{L^p(\mu)} + \Big ( \int_H \|D^2 v
(x)\|_{}^p \, \mu (dx)\Big)^{1/p}
+\|(-A)^{1/2}Dv\|_{L^{p}(\mu)}\le C\|f\|_{L^p(\mu)}.
$$
\end{theorem}

\section{Analytic results and an  It\^o type formula}

\subsection{ Existence and uniqueness for the Kolmogorov
equation when  $B$ is bounded}

Here we consider  the equation
\begin{equation}
\label{e12} \lambda u-L_2 \, u-\langle B,Du   \rangle=f,
\end{equation}
where $\lambda>0$, $f\in B_b(H)$ and $B\in B_b(H,H)$
 (i.e., $B: H \to H$ is Borel and bounded).

\begin{remark} {\rm Since
   the corresponding Dirichlet form
$$
{\mathcal E}(f, g) := \int_H \langle Df , Dg \rangle
 d\mu - \int_H \langle B, Df \rangle
 g
  \, d\mu  + \lambda \int_H f g
 \, d\mu,
$$
$f, \,   g \in W^{1,2}(\mu)$, is weakly sectorial for $\lambda$ big
enough,
  it follows by Chap. I and Subsection 3e in Chap. II of \cite{MR} that \eqref{e12} has a unique solution in $D(L_2)$. However, we  need
 more regularity for $u$.
}
\end{remark}
We  recall a result from  \cite{DFPR}.
\begin{proposition}
\label{p2} Let $\lambda\ge\lambda_0$, where
 \begin{equation} \label{dt1}
\lambda_0:=4\|B\|_0^2 C^2_{0}.
\end{equation}
Then there is a unique solution $u \in D(L_2)$ of \eqref{e12} given
by
$$
 u=  u_{\lambda}=(\lambda-L_2)^{-1}(I-T_\lambda)^{-1} f,
$$
where
$$
 T_\lambda g:=\langle B,D(\lambda-L_2)^{-1} g
\rangle.
$$
 Moreover, $u \in C^1_b (H)$ with
$$
\|u\|_0\le 2\|f\|_0,\quad\| Du\|_0\le
\frac{2C_{1,0}}{\lambda^{\frac12}}\; \|f\|_0,
$$
and,  for any $p \ge 2$, $u \in W^{2,p}(H, \mu)$ and,
 for some $ \, c=c(\lambda, p , \| B\|_0)$,
$$
  \int_H \|D^2 u (x)\|_{}^p \, \mu (dx)
\le c \int_H |f(x)|_{}^p \, \mu (dx).
$$
\end{proposition}
\subsection{Approximations}
We are given  two sequences $(f_n)\subset B_b(H)$ and   $(B_n)\subset B_b(H,H)$ and a constant $M>0$ such that
\begin{align} \label{e16}
\nonumber &  (i)\;\;\;f_n(x)\to f(x),\quad B_n(x)\to B(x)\quad\mu\mbox{\rm-a.e.}.
\\
& (ii) \;\;\; \|f_n\|_0\le M,\quad \|B_n\|_0\le M, \;\; \;\; n \ge 1.
\end{align} The following result has been proved in \cite{DFPR}.
\begin{proposition}
\label{p35} Let $\lambda \ge \lambda_0$,  where
    $\lambda_0$ is given in \eqref{dt1}.
 Then the equation
$$
\lambda u_n-L u_n-\langle B_n,Du_n   \rangle=f_n,
$$
has a unique solution $u_n\in C_b^1(H) \cap D(L_2)$ which is given by
$$
u_n=(\lambda-L)^{-1}(I-T_{n,\lambda})^{-1} f_n,
$$
where
$ T_{n,\lambda} \varphi:=\langle B_n, D(\lambda-L_2)^{-1}
\varphi\rangle.
$
Moreover,
\begin{equation}
\label{e20} \|u_n\|_0\le 2M,\quad \| Du_n\|_0 \le
\frac{2C_{1,0}}{\lambda^{\frac12}}M, \;\;\; n \ge 1.
 \end{equation}
Finally,
 we have $u_n\to u,$
 and
 $Du_n\to Du,$ in
 $L^2(\mu)$, where $u$ is the solution to \eqref{e12}.
\end{proposition}

Next we prove a  new result. The idea behind the result is  that
 if $(f_n)$ satisfies \eqref{e16} then, for any $x \in H$ (not only $\mu$-a.e.), $t>0$,
 \begin{align} \label{f5}
 R_t f_n (x) \to R_t f(x)
 \end{align}
as $n \to \infty$, due to the fact that, for any $x \in H,$ the law of the OU process $Z(t,x)$ at time $t>0$ is absolutely continuous with respect to $\mu$.


\begin{lemma}
\label{new11}  Consider the situation of  Proposition \ref{p35}. Then we have:
\begin{align}\label{pte}
 u_n (x) \to u(x), \;\;\;  Du_n (x) \to Du(x),
\end{align}
for any $x \in H$.
\end{lemma}
\begin{proof} By a standard argument, possibly  passing to a subsequence, we may assume  that $Du_n(x) \to Du(x)$, $\mu$-a.e.. It follows that for any $x$, $\mu$-a.e.,
$$
f_n(x) + \langle B_n(x), D u_n(x)\rangle \to f(x) + \langle B(x), D u(x) \rangle,
$$
as $n \to \infty$.  We write, for any $\lambda \ge \lambda_0$, $x \in H$,
$$
u_n(x) = \int_0^{\infty} e^{-\lambda t } R_t \big (f_n + \langle B_n, D u_n\rangle \big)(x)dt.
$$
By  the argument used in \eqref{f5} we can apply the dominated convergence theorem and obtain that
 $u_n (x) \to u(x)$, as $n \to \infty$, $x \in H$.

\smallskip Concerning $Du_n$, using \eqref{e6}, we obtain, for any $x \in H$, $h \in H$,
\begin{align*}
\langle Du_n(x), h \rangle = \int_0^{\infty} e^{-\lambda t } \langle DR_t \big (f_n + \langle B_n, D u_n\rangle \big)(x), h\rangle dt.
\end{align*}
Setting $g_n = f_n + \langle B_n, D u_n\rangle $,
 $g = f + \langle B, D u\rangle $,
note that
\begin{align*}
\langle DR_t g_n (x), e_k \rangle =  \int_H \langle
  \Lambda_t e_k,Q_t^{-\frac12} y\rangle \, g_n(e^{tA}x+y)N(0,Q_t)(dy),
  \;\; k \ge 1, \, x \in H, \, t>0.
\end{align*}
It follows that
$$
|\langle DR_t (g_n - g) (x), e_k \rangle |^2
\le |\Lambda_t e_k|^2  \,
 \int_H  |g(e^{tA}x+y) -  g_n(e^{tA}x+y)|^2 N(0,Q_t)(dy),
$$
and so
$$
|DR_t (g_n - g) (x)   |^2
\le \|\Lambda_t \|^2_{ }  \,
 \int_H  |g(e^{tA}x+y) -  g_n(e^{tA}x+y)|^2 N(0,Q_t)(dy).
$$
Now we get $|DR_t (g_n - g) (x)   |^2 \to 0$ as $n \to \infty$, for any $x \in H$, $t>0$, by the same argument used in \eqref{f5}.

Using again the dominated convergence theorem we obtain easily that $Du_n(x) \to Du(x)$, as $n \to \infty$, for any $x \in H$.
\end{proof}

\subsection{Modified mild formulation}
Recall   the notation
\begin{align}\label{bou1}
    B_N = B\, 1_{B(0,N)},\;\;\; N \ge 1,
\end{align}
where $B(0,N)$ is the open ball of radius $N$ (hence $B_N \in B_b(H,H)$, $N \ge 1$).

For any $i\in \mathbb N$ we denote  the $i^{th}$ component of $B$ by
$B^{(i)}$, i.e.,
$$
B^{(i)}(x):=\langle B(x),e_i  \rangle,\;\;\; x \in H.
$$
Then  for $\lambda \ge 4 \|B_N \|^2_0 C_0^2$  we consider the  solution
$u^{(i)}_N$ of the equation
\begin{equation}
\label{e21} \lambda u^{(i)}_N-Lu^{(i)}_N-\langle B_N, D u^{(i)}_N
\rangle=B^{(i)}_N,\quad \mu\;\mbox{\rm-a.e.}
\end{equation}
We recall that by Proposition \ref{p2},  $u^{(i)}_N \in C^1_b (H)$
 and,  for any $p \ge 2$, $u^{(i)}_N \in W^{2,p}(H, \mu).$

 The next result
  is a kind of ``local version'' of Theorem 7 in \cite{DFPR}.
 In contrast to Theorem \ref{main theorem} the result holds
for any initial condition $x \in H.$

\begin{remark} \label{mi1}
{\rm  Compared   with the proof of   Theorem 7 in \cite{DFPR}, here  we will   use  Lemma \ref{new11} which  allows to simplify some arguments of \cite{DFPR} and it is also
 needed to justify  the approximation procedure (see in particular \eqref{co5}).
 We also mention that differently with respect to \cite{DFPR} in
  Step 3 of the proof we need to construct a  suitable auxiliary process $\hat X^N= (\hat X_t^N)$
 (see \eqref{xht}) in order to  apply
  the Girsanov theorem and  get the assertion.
}
\end{remark}

\begin{theorem}
\label{p4} Let $X = (X_t)$ be  a weak mild solution of equation \eqref{uno}
 defined on some filtered probability space with
  a cylindrical $({\mathcal F}_{t})$-Wiener process $W$.
  Consider the stopping time
$$
\tau_N^X = \inf \{t \ge 0 \; :\; X_t \not \in B(0,N)\}
$$
 Let $u^{(i)}_N$ be the solution of \eqref{e21} and set
$X^{(i)}_t=\langle  X_t,e_i  \rangle$.
  For any $t>0$ we have $\mathbb \P\mbox{\rm
-a.s.}$ on the event
 $\{ t \le \tau_N^X \}$
\begin{align}
\label{e22}  \nonumber  X_t^{(i)} = \, & e^{-\lambda_i t}(\langle x,e_i
\rangle +u^{(i)}_N(x)) -u^{(i)}_N(X_t)
\\ \nonumber  &
+(\lambda+\lambda_i)\int_0^te^{-\lambda_i (t-s)}u^{(i)}_N(X_s)ds
\\
+&\int_0^te^{-\lambda_i (t-s)}(d\langle W_s,e_i \rangle+\langle
Du^{(i)}_N(X_s),dW_s  \rangle).\;
\end{align}
\end{theorem}

\begin{proof}
We fix $t>0$, $N \ge 1$ and $i \ge 1$.

{\bf Step 1} Approximation of $B_N$  and $u_N$.\medskip

Set
\begin{equation}
\label{e29} B_{N, n}(x)=\int_H B_N(e^{\frac1nA}x+y)N(0,
Q_{\frac1n})(dy),\quad x\in H.
\end{equation}
Then $B_{N, n }$ is of $C^\infty$ class and all its derivatives are
bounded. Moreover, $\|B_{N, n }\|_0\le \|B_N\|_0$, $n \ge 1$. It is easy to see that,
possibly passing to a subsequence,
\begin{equation}
\label{e30} B_{N, n }\to B_N,\quad \;\; \mu-a.e..
\end{equation}
(indeed $B_{N, n }\to B_N \;  \mbox{\rm in}\;L^2(H,\mu;H)$;  this result can
be first checked for continuous and bounded $B$). Now we denote by $u_{N, n }^{(i)}$ the solution of the equation
\begin{equation}
\label{e31}
\lambda u_{N, n }^{(i)}- Lu_{N, n }^{(i)}-\langle B_{N, n },Du_{N, n }^{(i)}   \rangle=B_{N, n }^{(i)},
\end{equation}
where $B_{N, n }^{(i)}=\langle B_{N, n },e_i  \rangle.$ By Lemma \ref{new11}
we have, possibly passing to a subsequence, for any $x \in H$,
\begin{align} \label{conv1}
\lim_{n\to \infty}u_{N, n }^{(i)}(x)=u^{(i)}_N(x), \;\; \lim_{n\to
\infty}Du_{N, n }^{(i)}(x)=D u^{(i)}_N(x),\;\;
\\ \nonumber \;\;
\sup_{n \ge 1} \| u_{N, n }^{(i)} \|_{C^{1}_b(H)} =
C_{ i, N} < \infty,
\end{align}
 where $u^{(i)}_N$  is the solution of \eqref{e21}.
\medskip

{\bf Step 2} Approximation of $X_t$.\medskip

For any $m\ge i$ we set $X_m = (X_{m,t})$, $X_{m,t}:=\pi_m X_t,$
where    $\pi_m=\sum_{j=1}^me_j\otimes e_j$. Then we have
\begin{equation}
\label{e32}
X_{m,t}=\pi_mx+\int_0^tA_mX_{s}ds+\int_0^t\pi_mB(X_{s})ds+\pi_mW_t,
\end{equation}
 where $A_m=\pi_mA$.

Now we denote by $u^{(i)}_{N, n,m}$ the solution of the equation
\begin{equation}
\label{e33}
\lambda u^{(i)}_{N,n,m}- L u^{(i)}_{N,n,m}-\langle \pi_m B_{N,n} \circ \pi_m, Du^{(i)}_{N,n,m}   \rangle=B_{N,n}^{(i)} \circ \pi_m,
\end{equation}
where  $(B_{N,n} \circ \pi_m) (x) = B_{N,n} ( \pi_m x)$, $x \in H$. Since only a
finite number of variables is involved, we have, equivalently,
$$
\lambda u^{(i)}_{N,n,m} - L^{(m)}u^{(i)}_{N,n,m}-\langle \pi_m B_{N,n} \circ
\pi_m, Du^{(i)}_{N,n,m}   \rangle=B_{N,n}^{(i)} \circ \pi_m,
$$
with
\begin{equation}
\label{e34}
L^{(m)}\varphi(x)=\frac12\;\mbox{\rm Tr}\;[\pi_mD^2\varphi(x)]+\langle A_mx,D\varphi(x)  \rangle.
\end{equation}
Moreover, since $u_{N,n,m}^{(i)}$ depends only on the first $m$ variables,  we have
\begin{align} \label{uu}
u_{N,n,m}^{(i)} (\pi_m y) = u_{N,n,m}^{(i)} (y), \;\;\;
y \in H.
\end{align}
Applying the finite-dimensional It\^o formula to $u^{(i)}_{N,n,m}(X_{m,t})
= u^{(i)}_{N,n,m}(X_{t})$  with the stopping time
 $\tau_N^X$
yields
\begin{gather}
\label{e35}
u^{(i)}_{N,n,m}(X_{m, {t \wedge \tau_N^X}
 }) - u^{(i)}_{N,n,m}(\pi_m x) =
\\ \nonumber
\int_0^{t \wedge \tau_N^X} \Big (\frac12\;\mbox{\rm Tr}\;[D^2u^{(i)}_{N,n,m}(X_{m,s})]
+\langle Du^{(i)}_{N,n,m}(X_{m,s}), A_m X_s+\pi_m B(X_s) \rangle \Big)ds
\\ \nonumber
+\int_0^{t \wedge \tau_N^X}  \langle Du^{(i)}_{n,m}(X_{m,s}), \pi_md W_s  \rangle.
\end{gather}
On the other hand, by \eqref{e33} we have
$$
\begin{array}{l}
\lambda u^{(i)}_{N,n,m}(X_{m,t})-\frac12\;\mbox{\rm Tr}\;[D^2u^{(i)}_{N,n,m}(X_{m,t})]\\
\\
-\langle Du^{(i)}_{N,n,m}(X_{m,t}), A_mX_{m,t}+\pi_mB_{N,n}(X_{m,t}) \rangle=B_{N,n}^{(i)}(X_{m,t}).
\end{array}
$$
 Let us fix $r \in ]0,t]$. This will be useful in Step 3 of the proof
to apply the Girsanov theorem (see in particular \eqref{f57}).

Comparing with \eqref{e35} and using \eqref{uu}
 we find
\begin{equation}
\label{e36}
 u^{(i)}_{N,n,m}(X_{ {t \wedge \tau_N^X}}) - u^{(i)}_{N,n,m}( X_{ {r \wedge \tau_N^X}})
\end{equation}$$
= \lambda \int_{r \wedge \tau_N^X}^{t \wedge \tau_N^X} u^{(i)}_{N,n,m}(X_{s})ds - \int_{r \wedge \tau_N^X}^{t \wedge \tau_N^X} B_{N,n}^{(i)}(X_{m,s})ds
$$$$
+\int_{r \wedge \tau_N^X}^{t \wedge \tau_N^X} \langle Du^{(i)}_{N,n,m}(X_{s}),  \pi_m(B(X_s)-B_{N,n}(X_{m,s})) \rangle ds\\
+\int_{r \wedge \tau_N^X}^{t \wedge \tau_N^X} \langle Du^{(i)}_{N,n,m}(X_{s}), d W_s  \rangle.
$$
Possibly passing to a subsequence,  and taking the limit in
probability as $m \to \infty$ (with respect to $\mathbb{P}$),
 we arrive at
\begin{gather} \label{fr3}
u^{(i)}_{N,n}(X_{{t \wedge \tau_N^X}}) - u^{(i)}_{N,n}(X_{{r \wedge \tau_N^X}})
 \\
\nonumber = \lambda \int_{r \wedge \tau_N^X}^{t \wedge \tau_N^X} u^{(i)}_{N,n}(X_{s})ds - \int_{r \wedge \tau_N^X}^{t \wedge \tau_N^X} B_{N,n}^{(i)}(X_{s})ds
\\ \nonumber
+\int_{r \wedge \tau_N^X}^{t \wedge \tau_N^X} \langle Du^{(i)}_{N,n}(X_{s}),  (B(X_s)-B_{N,n}(X_{s})) \rangle ds
+\int_{r \wedge \tau_N^X}^{t \wedge \tau_N^X} \langle Du^{(i)}_{N,n}(X_{s}), dW_s  \rangle.
\end{gather}
Let us   justify this  assertion.

First note that in equation \eqref{e33} we have the drift term
 $\pi_m B_{N,n} \circ \pi_m  $ which converges pointwise
 to $B_{N,n}$ and $B_{N,n}^{(i)} \circ \pi_m$ which
converges pointwise to $B_{N,n}^{(i)}$ as $m \to \infty$. Since such
functions are also uniformly bounded, we can apply Proposition
\ref{p35} and Lemma \ref{new11} and obtain that, possibly  passing to a subsequence (recall
that $n$ is fixed),
\begin{gather} \label{conv12}
\lim_{m\to \infty}u_{N,n,m}^{(i)}(x)=u^{(i)}_{N,n}(x), \;\; \lim_{m\to
\infty}Du_{N,n,m}^{(i)}(x)=D u^{(i)}_{N,n}(x),
\;\; x \in H,
\\
\nonumber \;\; \sup_{m \ge 1} \| u_{N,n,m}^{(i)} \|_{C^{1}_b(H)} =
C_{ i}^N < \infty.
\end{gather}
We only consider convergence of the two most involved terms in \eqref{e36}.

We first treat  convergence in $L^2(\Omega) $ of the stochastic integral. Recall that
$$
\int_0^{t \wedge \tau_N^X} \langle Du^{(i)}_{N,n}(X_{s}), dW_s  \rangle = \int_0^{t } 1_{\{ s \le  t \wedge \tau_N^X\} } \langle Du^{(i)}_{N,n}(X_{s}), d W_s  \rangle;
$$
by the isometry formula and \eqref{conv12} we get
\begin{equation} \label{co5}
 \E \Big | \int_0^{t \wedge \tau_N^X} \langle Du^{(i)}_{N,n,m}(X_{s})-  Du^{(i)}_{N,n}(X_{s}) , d W_s  \rangle \Big|^2  \to 0
\end{equation}
as $m \to \infty.$  Note that we have used  that
 $\lim_{m\to
\infty}Du_{N,n,m}^{(i)}(x)=D u^{(i)}_{N,n}(x),$ for any
 $ x \in H$ (not only for   $\mu$-a.e $x \in H$).
In a similar way we get
 $$
 \E \Big | \int_0^{r \wedge \tau_N^X} \langle Du^{(i)}_{N,n,m}(X_{s})-  Du^{(i)}_{N,n}(X_{s}) , dW_s  \rangle \Big|^2  \to 0
$$
as $m \to \infty.$

This shows that  $\int_{r \wedge \tau_N^X}^{t \wedge \tau_N^X} \langle Du^{(i)}_{N,n,m}(X_{s}), d W_s  \rangle$ converges  to $\int_{r \wedge \tau_N^X}^{t \wedge \tau_N^X} \langle Du^{(i)}_{N,n}(X_{s}), d W_s  \rangle$ in $L^2(\Omega)$ as $m \to \infty.$ To show that, $\P$-a.s.,
\begin{gather} \label{dop}
\lim_{m \to \infty}
 \int_{r \wedge \tau_N^X}^{t \wedge \tau_N^X} \Big|\langle Du_{N,n,m}^{(i)}(X_{s}),\pi
_{m}(B(X_{s})-B_{N,n}(X_{m,s}))\rangle\\ \nonumber - \langle Du_{N, n}
^{(i)}(X_{s}),(B(X_{s})-B_{N,n}(X_{s}))\rangle \Big| ds
 =0,
\end{gather}
 it is enough to  prove that $\lim_{m \to \infty} H_m + K_m =0$, where
$$
H_m =
 \int_{0}^{t \wedge \tau_N^X} \big|\langle Du_{N,n,m}^{(i)}(X_{s})- Du_{N,n}^{(i)}(X_{s}) ,\pi
_{m}(B(X_{s})-B_{N,n}(X_{m,s}))\rangle \big| ds
$$
and
$$
K_m =
 \int_{0}^{t \wedge \tau_N^X} \big|\langle Du_{N,n}^{(i)}(X_{s}),
 [\pi_{m}B(X_{s}) - B(X_{s})] + [B_{N,n}(X_{s}) - \pi_m B_{N,n}(X_{m,s})]
 \rangle \big| ds.
$$
By using \eqref{conv12} we easily get  the assertion.

\smallskip
{\bf Step 3} A convergence result involving stopping times.\medskip

 In order to pass to
 the limit in probability  as $n \to \infty$ in \eqref{fr3} we recall formula
\eqref{conv1} and argue as before. The only difficult term is
$$\int_{r \wedge \tau_N^X}^{t \wedge \tau_N^X} \langle Du^{(i)}_{N,n}(X_{s}),  (B(X_s)-B_{N,n}(X_{s})) \rangle ds = J_n + I_n,
 $$
$$
\text{where} \; J_n = \int_{r \wedge \tau_N^X}^{t \wedge \tau_N^X} \langle Du^{(i)}_{N,n}(X_{s})- Du^{(i)}_{N}(X_{s}),  (B(X_s)-B_{N,n}(X_{s})) \rangle ds, $$$$ I_n =
\int_{r \wedge \tau_N^X}^{t \wedge \tau_N^X} \langle Du^{(i)}_{N}(X_{s}),  (B_N(X_s)-B_{N,n}(X_{s})) \rangle ds
$$
(using that $s \le t \wedge \tau_N^X$).  As for $J_n$ we have
$$
J_n = \int_{r \wedge \tau_N^X}^{t \wedge \tau_N^X} \langle Du^{(i)}_{N,n}(X_{s})- Du^{(i)}_{N}(X_{s}),  (B_N(X_s)-B_{N,n}(X_{s})) \rangle ds,
$$
and so
$
|J_n| \le 2 \| B_N\|_0 \, \int_0^{t } |Du^{(i)}_{N,n}(X_{s})- Du^{(i)}_{N}(X_{s})|ds \to 0
$, $\P$-a.s.,
as $n \to \infty$, by Lemma \ref{new11}.

Let us consider $I_n$. We define  an auxiliary process $\hat X^N = (\hat X_t^N)$ as follows:
\begin{equation}\label{xht}
  \hat X_t^N := e^{tA}x + \int_{0}^t e^{(t-s)A} B_N (X_{s\wedge \tau_N^X})ds + \int_0^t e^{(t-s)A} dW_s, \;\; t \ge 0,
\end{equation}
Note that  $X_{s \wedge \tau^X_N} = \hat X^N_{s \wedge \tau^X_N}$,
 $s \ge 0$, so that
\begin{equation} \label{ret}
|I_n | = \Big | \int_{r \wedge \tau_N^X}^{t \wedge \tau_N^X} \langle Du^{(i)}_{N}(\hat X_{s}^N),  (B_N(\hat X_{s}^N)-B_{N,n}(\hat X_{s}^N)) \rangle ds\Big |$$$$ \le \| D^{(i)}u_N \|_0  \int_r^{t } | B_N(\hat X_{s}^N)-B_{N,n}(\hat X_{s}^N)| ds.
\end{equation}
Now we use the Girsanov theorem (see e.g. Appendix in \cite{DFPR}). Let $T>0$.
Since
$$
 \hat X_t^N := e^{tA}x + \int_{0}^t e^{(t-s)A} \hat B_s^N ds + \int_0^t e^{(t-s)A} dW_s, \;\; t \ge 0,
$$
where $\hat B_s^N  = B_N (\hat X_{s\wedge \tau_N^X}^N)$, $s \ge 0$, is an adapted and bounded process,
 we have that
$$
\tilde W_{t}^N:=W_{t}+\int_{0}^{t}\hat{B}_{s}^{N}ds
$$
is a cylindrical Wiener process on $\left(  \Omega, {\mathcal F},\left( {\mathcal F}_{t}\right) _{t\in\left[  0,T\right]},\widetilde{\mathbb{P}}_{N}\right)  $ where $\left.  \frac{d\widetilde{\mathbb{P}}_{N}}
{d\mathbb{P}}\right\vert _{{\mathcal F}_{T}}=\rho_{N}$,
$$
\rho_{N}=\exp
\left(-\int_{0}^{T}\hat {B}_{s}^{N}dW_{s}-
\frac{1}{2}\int_{0}^{T}|\hat{B}_{s}^{N}|^{2}ds\right).
$$
Hence $ \hat X_t^N =e^{tA}x+\int_{0}^{t}e^{\left(  t-s\right)  A}d\tilde W_{s}^N$  is an OU process on
 $(\Omega, {\mathcal F},  ({\mathcal F_t})_{t\in\left[  0,T\right]}, \tilde {\mathbb P}_{N})$.

Moreover, we  know that the law of $(\hat X_t^N)_{t \in [0,T]} $ on $C([0,T]; H)$    is equivalent to the law of the OU   process $Z(t,x)$  given in \eqref{ou11}.
 In particular, all their transition probabilities are equivalent. Now  under our assumptions   the law of $Z(t,x)$ is equivalent to $\mu$ for all $t >0 $ and $x \in H$   (see Theorem 11.3 in \cite{DZ}).

Let us come back to \eqref{ret}. Using that the law $\pi_t^N(x, \cdot)$ of $\hat X_t^N$ is absolutely continuous with respect to $\mu$, we obtain
\begin{equation} \label{f57}
    \mathbb{E}
\int_{r}^{t} | B_N(\hat X_{s}^N)-B_{N,n}(\hat X_{s}^N)| ds
= \int_{r}^{t} ds \int_H \big| B_N (y)- B_{N,n}(y)\big| \, \frac{d \pi_s^N(x,\cdot)}{d \mu}(y) \mu (dy),
\end{equation}
which tends to 0, as $n \to \infty$, by the dominated convergence theorem.
 Hence we have  found that
 $I_n \to 0$  in   $L^{1}(\Omega, \mathbb{P}).$

 Up to now  we have
\begin{gather*}
u^{(i)}_{N}(X_{{t \wedge \tau_N^X}
 }) - u^{(i)}_{N}(X_{r \wedge \tau_N^X})
\\
= \lambda \int_{r \wedge \tau_N^X}^{t \wedge \tau_N^X} u^{(i)}_{N}(X_{s})ds - \int_{r \wedge \tau_N^X}^{t \wedge \tau_N^X} B_{}^{(i)}(X_{s})ds
+\int_{{r \wedge \tau_N^X}}^{t \wedge \tau_N^X} \langle Du^{(i)}_{N}(X_{s}), dW_s  \rangle.
\end{gather*}
Passing to the limit as $r \to 0^+$, since the trajectories of $X$ are continuous, we finally get
\begin{gather} \label{f54}
u^{(i)}_{N}(X_{{t \wedge \tau_N^X}
 }) - u^{(i)}_{N}(x)
 \\ \nonumber
= \lambda \int_0^{t \wedge \tau_N^X} u^{(i)}_{N}(X_{s})ds - \int_0^{t \wedge \tau_N^X} B_{}^{(i)}(X_{s})ds
+\int_0^{t \wedge \tau_N^X} \langle Du^{(i)}_{N}(X_{s}), dW_s  \rangle.
\end{gather}
{\bf Step 4}  The final formula.\medskip

By \eqref{uno} we deduce
$$
dX^{(i)}_t=-\lambda_iX^{(i)}_tdt+B^{(i)}(X_t)dt+dW_t^{(i)}.
$$
Inserting the expression for  $B^{(i)}(X_t)$, which  we get from this identity,
 into \eqref{f54}, we obtain
\begin{gather*}
u^{(i)}_{N}(X_{{t \wedge \tau_N^X}
 }) - u^{(i)}_{N}(x)
\\
= - X_{{t \wedge \tau_N^X}}^i + x^i
+
\lambda \int_0^{t \wedge \tau_N^X} u^{(i)}_{N}(X_{s})ds  -\lambda_i \int_0^{t \wedge \tau_N^X} X^{(i)}_s ds
+\int_0^{t \wedge \tau_N^X} \langle Du^{(i)}_{N}(X_{s}), d W_s  \rangle \\ + W^{i}_{t \wedge \tau_N^X}.
\end{gather*}
By the variation of constants formula this is equivalent to
\begin{gather*}
X^{(i)}_{t \wedge \tau_N^X}=\displaystyle e^{-\lambda_i {t \wedge \tau_N^X}}\langle x,e_i
\rangle+\lambda\int_0^{t \wedge \tau_N^X}e^{-\lambda_i ({t \wedge \tau_N^X}-s)}u^{(i)}_N(X_s)ds \\
\displaystyle - \int_0^{t \wedge \tau_N^X} e^{-\lambda_i ({t \wedge \tau_N^X} -s)}du^{(i)}_N(X_s)
\displaystyle +
e^{-\lambda_i ({t \wedge \tau_N^X})}
\int_0^{t \wedge \tau_N^X} e^{\lambda_i s}[dW_s^{(i)}+\langle
Du^{(i)}_N(X_s),dW_s \rangle].
\end{gather*}
This identity  yields  \eqref{e22} on $\{ t \le  \tau_N^X \}$.
 \end{proof}

The next lemma is similar to Lemma 9 in \cite{DFPR}  and shows that $u_N(x) = \sum_{k \ge 1} u^{(k)}_N(x) e_k$
(where $u^{(k)}_N$ is as in \eqref{e21}) is a well defined function which
belongs to $C^1_b(H,H)$.
\begin{lemma} \label{de}
 For  $\lambda $
sufficiently large, i.e., $\lambda\ge \tilde \lambda$,
 with $\tilde \lambda = \tilde \lambda (A, \| B_N\|_0)>0$
there exists a unique $u_N = u_{\lambda, N} \in C^{1}_b(H,H)$ which
solves
$$
u_N(x) = \int_0^{\infty} e^{-\lambda t } R_t \big(Du_N(\cdot) B_N(\cdot) +
B_N(\cdot)\big)(x)dt,\;\;\; x \in H,
$$
where $R_t$ is the OU semigroup defined as in \eqref{e1} and acting
on $H$-valued functions. Moreover, we have the following assertions.

(i)   For any $h \in H$,
$Du_N(\cdot)[h]
 \in C_b (H, H)$   and $\|Du_N(\cdot)[h] \|_0
 $ $ \le C_{0, \lambda, N} |h|$;

(ii) for any $k \ge 1$, $\langle u_N(\cdot), e_k\rangle = u^{(k)}_N$,
  where $u^{(k)}_N$ is the solution defined in \eqref{e21};

(iii) There exists $c_{3}= c_3(A, \|B_N \|_0)>0$ such that, for any
$\lambda \ge \tilde \lambda$, $u = u_{\lambda}$ satisfies
 \begin{align} \label{grado}
 \| Du_N \|_0 \le \frac{c_3}{\sqrt{\lambda}}.
\end{align}
\end{lemma}


\section{Proof  of Theorem \ref{main theorem}}
\label{section proof}
 Let $X= (X_{t})$ and $Y = (Y_{t})$ be two  weak
  mild solutions (see \eqref{ci71} and \eqref{ci72}) defined on the same filtered
 probability space (solutions with respect to the same cylindrical Wiener process
 $W$)  starting at   $x \in H$.

 In the first part of the proof
 we will adapt  the proof of Theorem 1 in \cite{DFPR} by introducing additional
 stopping times. The main difference with respect to \cite{DFPR}
 will appear in Proposition
 \ref{stop} which is needed to finish the proof.

For the time being, $x$ is not specified.
In Proposition \ref{stop} a
restriction on $x$ will emerge.

Note that by our hypothesis
\begin{equation} \label{vu}
B_N = B 1_{B(0,N)} =  B' 1_{B(0,N)}=  B'_N.
\end{equation}
It follows that  Kolmogorov equation \eqref{e21} written with respect to
the truncated drift $B_N$ (with $B_N^{(i)}$ in the right-hand side) or
with respect to $B'_N$
 (with $B_N^{' \, (i)}$ in the right-hand side)
is the same and gives  the same solution $u^{(i)}_{N, \lambda}$.

It follows that both $X$ and $Y$ satisfy \eqref{e22} on the event
 $\{ t \le \tau_N^X \wedge \tau_N^Y \}$.

Now we  consider
$$
u_N = u_{N, \lambda}:  H \to H
$$  be  such that
   $u_N(x) = \sum_{i \ge 1} u^{(i)}_N(x)e_i$, $x \in H$, where
 $u^{(i)}_N = u^{(i)}_{N, \lambda} $
  solve \eqref{e21}   for some   $\lambda$
  large enough possibly depending on $N$.

 Let us fix $T>0$.
  By  \eqref{e22}, taking into account \eqref{vu}, we have, for $t \in [0,T\wedge \tau_N^X \wedge \tau_N^Y]$,
  $\P$-a.s.,
\begin{align*}
X_{t}-Y_{t} &  =u_N\left(  Y_{t}\right)  -u_N\left(  X_{t}\right)
+\left( \lambda-A\right)  \int_{0}^{t}e^{\left(  t-s\right) A}\left(
u_N\left( X_{s}\right)  -u_N\left(  Y_{s}\right)  \right) ds
 \\
&  + \int_{0}^{t}e^{\left(  t-s\right)  A} (  Du_N\left( X_{s}\right)
-Du_N(  Y_{s}))    dW_{s}.
\end{align*}
  Here and  in the sequel we will drop the $ \lambda$-dependence of $u_N$ to
  simplify notation.
 However,  at the end we will fix a value of $\lambda$ large enough.
  By \eqref{grado} we may assume
  that
 $ \| Du_N \| _{0}\leq 1/2.$

It follows that  for $t \in [0,T\wedge \tau_N^X \wedge \tau_N^Y]$,
\begin{align*} \left\vert X_{t}-Y_{t}\right\vert
\leq\frac{1}{2}\left\vert X_{t} -Y_{t}\right\vert  + \Big |  \left(
\lambda-A\right) \int_{0}^{t}e^{\left( t-s\right) A}\left(  u_N\left(
X_{s}\right) -u_N\left(  Y_{s}\right) \right) ds \Big|
\\
+ \Big | \int_{0}^{t}e^{\left(  t-s\right)  A}
 (  Du_N\left(
X_{s}\right) -Du_N(  Y_{s}))    dW_{s}  \Big|.
\end{align*}
Let $\eta$ be a  stopping time to be specified later and set
\begin{equation} \label{r88}
\tau = \eta \wedge T\wedge \tau_N^X \wedge \tau_N^Y.
\end{equation}
  Using that $1_{[0, \tau]}(t)
 = 1_{[0, \tau]}(t) \cdot $ $ 1_{[0,  \tau]}(s)$,
  $0 \le s \le t \le T$,
 we have (cf.  page 187 in \cite{DZ})
\begin{align*}
& 1_{[0, \tau]}(t)  \left\vert X_{t}-Y_{t}\right\vert \leq C \,
1_{[0, \tau ]}(t) \Big |\left( \lambda-A\right)
\int_{0}^{t}e^{\left( t-s\right) A}\left( u_N\left( X_{s}\right)
-u_N\left(  Y_{s}\right) \right) ds \Big|
\\
& + C\Big\vert 1_{[0, \tau]}(t)  \int_{0}^{t}e^{\left( t-s\right)
A}\left(  Du_N\left(  X_{s}\right) -Du_N\left(  Y_{s}\right) \right) \,
1_{[0, \tau ]}(s)\, dW_{s}
\Big\vert,
\end{align*}
where by $C$ we denote any constant which may depend on the
hypotheses on $A$, $B_N$ and  $T$.

Setting $1_{[0,\tau]}(s)\,X_s$ $=\tilde{X}_s$ and $1_{[0,\tau]}%
(s)\,Y_{s}=\tilde{Y}_{s}$, and, using the Burkholder-Davis-Gundy
inequality with   $q>2$ which will be determined
below, we obtain (recall that $\Vert \, \cdot\, \Vert_{}$ is  the
Hilbert-Schmidt norm, cf. Chapter 4 in \cite{DZ}) with $C = C_q$,
\begin{gather*}
  {\mathbb E} [  | \tilde{X}_{t}-\tilde{Y}_{t} |^{q}] \leq C_{}\, {\mathbb E}\Big [ e^{\lambda q t}
 \Big |  \left ( \lambda-A\right)
 \int_{0}^{t}e^{\left( t-s\right) A} e^{- \lambda s} ( u_N(  X_{s}) -u_N(  Y_{s}) )
 1_{[0,\tau]}(s) ds
\Big |^q \Big]
\\
  +\, C_{}{\mathbb E}\Big[  \Big(  \int_{0}^{t}1_{[0,\tau]}(s)
\left\Vert e^{\left( t-s\right)  A} (  Du_N\left(  X_{s}\right)
-Du_N\left( Y_{s}\right) )  \,\right\Vert _{}^{2}ds\Big)
^{q/2}\Big].
\end{gather*}
In the sequel we also consider a parameter $\theta>0$;  moreover, $C_{\theta}$
will denote suitable positive constants such that $C_{\theta}\to 0$ as $\theta\to +\infty$.
 Similarly,
 $C(\lambda)$ will denote suitable constants (possibly depending on $N$)
  such that $C(\lambda)\to 0$ as
  $\lambda\to \infty$.
 
From the previous inequality we deduce, multiplying by $ e^{-q\theta
t}$, for any $\theta >0,$
\begin{gather} \label{serve}
 {\mathbb E}\big[  e^{-q\theta t} | \widetilde{X}_{t}-\widetilde{Y}
_{t} |^{q}\big ]  \\ \nonumber
 \leq C \, {\mathbb E}\Big[  \Big\vert \left(  \lambda-A\right)  \int_{0}^{t}
e^{-\theta\left(  t-s\right)  }e^{\left(  t-s\right)  A}\left(
u_N\left( X_{s}\right)  -u_N\left(  Y_{s}\right)  \right)  e^{-\theta
s}1_{\left[
0,\tau\right]  }\left(  s\right)  ds 
\Big\vert ^{q}\Big]  \\
\nonumber  +\, C \, {\mathbb E}\Big[  \Big(
\int_{0}^{t}e^{-2\theta\left( t-s\right) }\left\Vert e^{\left(
t-s\right)  A}\left(  Du_N\left( X_{s}\right) -Du_N\left( Y_{s}\right)
\right)  \right\Vert ^{2}e^{-2\theta s}1_{\left[  0,\tau\right]
}\left(  s\right) ds\Big) ^{q/2}\Big]  .
\end{gather}
Now proceeding as in the proof of Theorem 7 of \cite{DFPR} we arrive at
\begin{gather} \label{serve2}
 \int_0^T{\mathbb E}\Big[  e^{-q\theta t}\left\vert \widetilde{X}_{t}-\widetilde{Y}
_{t}\right\vert ^{q}\Big] dt  \\ \nonumber
  \leq C(\lambda)\, \int_{0}^{T}e^{-q\theta s}{\mathbb E}|\tilde{X}_{s}-\tilde{Y}
_{s}|^{q} \, ds+ \tilde C_{\theta} \, {\mathbb
E} \Big[ \Lambda_T \cdot \int_{0}^{T}e^{-q\theta s}|\tilde{X}_{s}-\tilde {Y}_{s}|^{q}ds \Big]
\end{gather}
 provided that    $q \in (4, \infty)$,
  $\gamma = q/2$, $\theta \ge \lambda$
 and
\begin{gather} \label{ztr}
\nonumber \Lambda_{T}  
:=\int_{0}^{T} dt \int_{0}^{t}1_{[0,\tau]}(s)
\,ds \int_{0}^{1}\Big(
\sum_{n\geq1}\frac{1}{\lambda_{n}^{{1 - \delta}}}\Vert D^{2}u^{(n)}_N(Z_{s}^{r}
)\Vert^{2}\Big)  ^{\gamma}\,dr,
\\  \text{where} \;  Z_{t}^{r} =
 rX_{t}+(1-r)Y_{t},
 \end{gather}
 and $D^{2}u^{(n)}_N(x)$ is defined for $\mu-$a.e. $x \in H$. The existence of
 $D^{2}u^{(n)}_N \in L^p(\mu)$, $p \ge 2$,
  follows from Proposition \ref{p2} applied to
 equation \eqref{e21} (see also Lemma 23 in \cite{DFPR}).

 Since
 \begin{align*}
\Lambda_{T}    \leq T\cdot\int_{0}^{T\wedge\tau}\, ds \int_{0}^{1}\Big(
\sum_{n\geq1}\frac {1}{\lambda_{n}^{{1 - \delta}}}\Vert
D^{2}u^{(n)}_N(Z_{s}^{r})\Vert^{2}\Big) ^{\gamma}\,dr
\end{align*}
 it is natural to define, for any $R>0$, the stopping time
\[
 \bar \tau_{R}^{x, N}\, =\inf\left\{  t\ge 0  \, : \,  \int_{0}^{t}\,ds \int_{0}%
^{1}\Big(  \sum_{n\geq1}\frac{1}{\lambda_{n}^{{1 - \delta}}}\Vert D^{2}
u^{(n)}_N(Z_{s}^{r})\Vert^{2}\Big)^{\gamma}\,dr \geq
R\right\} \, \wedge \, T.
\]
 Take $\eta=\bar \tau_{R}^{x, N}$
in the previous expressions so that
$$
\tau =   \bar \tau_{R}^{x, N} \wedge \tau_N^X \wedge \tau_N^Y.
$$
 We get from \eqref{serve2}
\begin{gather*}
  \int_{0}^{T}e^{-q\theta t}{\mathbb E}|\tilde{X}_{t}-\tilde{Y}_{t}|^{q}dt\\
 \leq C(\lambda)\, \int_{0}^{T}e^{-q\theta s}{\mathbb E}|\tilde{X}_{s}-\tilde{Y}
_{s}|^{q}ds+ \tilde C_{\theta}\, R\, \int_{0}^{T}e^{-q\theta s}{\mathbb
E}|\tilde{X}_{s}-\tilde {Y}_{s}|^{q}ds.
\end{gather*}
Now we fix $\lambda$ large enough such that $ C (\lambda)  <1/2$. For sufficiently large $\theta=\theta_{R} \ge \lambda$, \textit{depending on
}$R$ and $N$, we have $\tilde C_{\theta}R <1/2$ and so
\[
{\mathbb E}\Big[\int_{0}^{T}e^{-q\theta_{R}\,t}
\, 1_{[0,\tau_{}]}(t)\,|X_{t}-Y_{t}
|^{q}dt\Big]={\mathbb E}\Big[\int_{0}^{\tau_{}}e^{-q\theta_{R}\,t}\,|X_{t}-Y_{t}%
|^{q}\, dt\Big]=0.
\]
In other words, for every $R>0$, $N \ge 1,$ $\mathbb{P}$-a.s., $X=Y$ on
$\big[ 0, \bar \tau_{R}^{x, N} \wedge \tau_N^X \wedge \tau_N^Y \big]  $ (identically in $t$, not only a.e. in $t$, since $X$ and
$Y$ are continuous processes).

If we prove that
\begin{align} \label{gol}
\lim_{R\rightarrow\infty}\bar \tau_{R}^{x, N}=
T \wedge \tau_N^X \wedge \tau_N^Y,  \;\;\; \mathbb{P}-a.s.,
\end{align}
 then we obtain that
$ X = Y  \;\; \text{on} \;\; [0, T \wedge \tau_N^X \wedge \tau_N^Y]$ and this
 finishes
the proof.


\smallskip
The crucial assertion \eqref{gol}  follows by the next proposition.


\begin{remark} \label{mi2}{\em
 Assertion \eqref{gol}  is a  ``local version'' of Proposition 10 in \cite{DFPR}.
Similarly to \eqref{xht} also in the next proof we have to  find an auxiliary process (see \eqref{r8}) which allows  to apply the Girsanov theorem.
}
\end{remark}

\begin{proposition}
\label{stop}  Let $N \ge 1 $ and $T>0$
 and suppose $X$ and $Y$ as in Theorem \ref{main theorem}. For $\mu$-a.e. $x\in H$, we have
   $\mathbb{P}\big(S_{T\wedge \tau_N^X \wedge \tau_N^Y}^x <\infty\Big)=1$, where
 \[
S_{t}^{x}=S_{t}^{x,N}=\int_{0}^{t}\,\int_{0}^{1}\Big( \sum_{n\geq1}\frac
{1}{\lambda_{n}^{{1 - \delta}}}\Vert
D^{2}u^{(n)}_N(Z_{s}^{r})\Vert^{2}\Big) ^{\gamma}\,dr\;ds, \;\;\; t \in [0,T],
\]
 with $\gamma = q/2$ ($
   u_N(x) = \sum_{i \ge 1} u^{(i)}_N(x)e_i$, $x \in H$, where
 $u^{(i)}_N = u^{(i)}_{N, \lambda} $
  solves \eqref{e21}$)$.
\end{proposition}
  \begin{proof}
To prove the assertion we will show that, for $\mu$-a.e. $x\in H$,
\[
{\mathbb E}[S_{T\wedge \tau_N^X \wedge \tau_N^Y}^{x}]<+\infty.
\]
 In the first part of the proof, $x\in H$ is given, without restriction.  Let us consider stopped processes
$$
 X_t^N = X_{t\wedge \tau_N^X \wedge \tau_N^Y},
\;\;\;   Y_t^N = Y_{t\wedge \tau_N^X \wedge \tau_N^Y},
$$
and then we define an auxiliary process  $( Z_t^{r,N})_{t \in [0,T]}$ as follows
 \begin{equation} \label{r8}
 Z_{t}^{r, N}:=e^{tA}x+\int_{0}^{t}e^{\left(  t-s\right)  A}\bar{B}_{s}^{r, N}
ds+\int_{0}^{t}e^{\left(  t-s\right)  A}dW_{s}%
\end{equation}
where  (recall \eqref{d9})
\[
\bar{B}_{s}^{r, N}:=[r B( X_{s}^N)+(1-r)B(Y_{s}^{N})],\;\;\;r\in\lbrack0,1],\;\; s \in [0,T].
\]
Comparing $Z^{r}$ (see \eqref{ztr}) and $Z^{r,N}$ we see that  $Z_{s \wedge \tau_N^X \wedge \tau_N^Y }^r = Z_{s \wedge \tau_N^X \wedge \tau_N^Y}^{r,N}$, $s \in [0,T]$, $r \in [0,1]$.
 Hence  we have to prove
$$
  \mathbb{E}[ S_{T\wedge \tau_N^X \wedge \tau_N^Y}^{x}]=\mathbb{E} \int_{0}^{T\wedge \tau_N^X \wedge \tau_N^Y}\,\int_{0}^{1}\Big( \sum_{n\geq1}\frac
{1}{\lambda_{n}^{{1 - \delta}}}\Vert
D^{2}u^{(n)}_N(Z_{s}^{r, N})\Vert^{2}\Big) ^{\gamma}\,dr\;ds < \infty.
$$
This follows if we can show that
\begin{equation}\label{te2}
 \mathbb{E} \int_{0}^{T}\,\int_{0}^{1}\Big( \sum_{n\geq1}\frac
{1}{\lambda_{n}^{{1 - \delta}}}\Vert
D^{2}u^{(n)}_N(Z_{s}^{r, N})\Vert^{2}\Big) ^{\gamma}\,dr\;ds < \infty.
\end{equation}
 We fix $N \ge 1$.
To verify \eqref{te2} we can   follow
the proof  of Proposition 10 in \cite{DFPR}. We only indicate the small changes which are needed.

Define
\[
\rho_{r, N}=\exp\left(
-\int_{0}^{t}\bar{B}_{s}^{r, N}dW_{s}-\frac{1}{2}\int
_{0}^{t}|\bar{B}_{s}^{r, N}|^{2}ds\right)  .
\]
We have, since $|\bar{B}_{s}^{r}|\leq\Vert B_N\Vert_{0}$, $r \in [0,1]$, $s \ge 0$, $\P$-a.s.,
\begin{equation}
{\mathbb E}\left[  \exp\left(
k\int_{0}^{T}|\bar{B}_{s}^{r, N}|^{2}ds\right) \right]
\leq C_{k}<\infty,\label{estimate easy}%
\end{equation}
for all $k\in\mathbb{R}$, independently of $x$ and $r$, simply
because $B_N$ is bounded. Hence an infinite dimensional version of
Girsanov's Theorem with respect to a cylindrical Wiener process (the
proof of which is included in the Appendix of \cite{DFPR})
applies and gives us that
\[
\tilde W_{t}^N:=W_{t}+\int_{0}^{t}\bar{B}_{s}^{r,N}ds
\]
is a cylindrical Wiener process on $\left(  \Omega, {\mathcal F},
\left( {\mathcal F}_{t}\right) _{t\in\left[  0,T\right]
},\widetilde{\mathbb{P}}_{r,N}\right)  $ where
$\left.  \frac{d\widetilde{\mathbb{P}}_{r,N}}{d\mathbb{P}}\right\vert _{{\mathcal F}_{T}%
}=\rho_{r,N}$. Hence
\[
Z_{t}^{r,N}=e^{tA}x+\int_{0}^{t}e^{\left(  t-s\right)  A}d\tilde W_{s}^N%
\]
is  an OU process on $\left(  \Omega, {\mathcal F},
\left( {\mathcal F}_{t}\right) _{t\in\left[  0,T\right]
},\widetilde{\mathbb{P}}_{r,N}\right)  $.
 Continuing as in the proof of Proposition 10 in \cite{DFPR} with $u^{(n)}$ replaced by $u^{(n)}_N$, $\rho_r$ replaced by $\rho_{r,N}$, $Z^r_s$ replaced by $Z^{r,N}_s$, $\bar{B}_{s}$ replaced by $\bar{B}_{s}^{r,N}$,  we see that \eqref{te2} holds if we prove that
\begin{align} \label{df}
\int_{0}^{T}{\mathbb E}\,\Big[  \Big(  \sum_{n\geq1}\frac{1}{\lambda_{n}^{{1 - \delta}}%
}\Vert D^{2}u^{(n)}_N(e^{sA}x+W_{A}(s))\Vert^{2}\Big)
^{2\gamma}\Big] ds<\infty,
\end{align}
 where $W_A(t) = \int_0^t e^{(t-s)A} dW(s), \;\; t \ge 0.$
If  $\mu_{s}^{x}$ denotes the law of $e^{sA}x+W_{A}\left( s\right)
$, we
have to prove that%
\begin{align} \label{cruciale}
\int_{0}^{T}\int_{H}\,\Big( \sum_{n\geq1}\frac{1}{\lambda_{n}^{{1 -
\delta}}}\Vert D^{2}u^{(n)}_N(y)\Vert^{2}\Big)
^{2\gamma}\mu_{s}^{x}(dy)ds<\infty.
\end{align}
This can be checked as in the mentioned proof (see in particular Steps 3 and  4 in that proof) {\it  only for $\mu$-a.e. $x \in H$};  one has   to replace $B$ in the proof in \cite{DFPR} with our $B_N$.
\end{proof}

\begin{remark} {\rm  As is easily checked in Theorem \ref{main theorem} the ball $B(0,N)$ can be replaced by any open bounded set in $H$.}
\end{remark}

\begin{remark} {\em
According to Remark 11 in \cite{DFPR} our Theorem \ref{main theorem} provides an alternative approach to  Veretennikov's uniqueness result in finite dimension.
 In this respect first note that Theorem \ref{main theorem}
  when $H = \R^d$ does not require to start from $\mu$-a.e. initial conditions $x$, but works for any initial $x \in \R^d$.

 Note that
 in finite dimension an
  SDE like $dX_t = b(X_t)dt + dW_t$
  with $b$ Borel and bounded is equivalent to
  $dX_t = - X_t dt +  (b(X_t)+ X_t) dt + dW_t$ which is in the form \eqref{uno}
 with $A = -I,$  and with a   drift term $B(x) = b(x) + x$ which is completely covered by Theorem \ref{main theorem}.

Recall that in this alternative approach  to Veretennikov's result,   basically the elliptic $L^p$-estimates with respect to  Lebesgue measure used in \cite{Ver}  are replaced by elliptic $L^p{(\mu)}$-estimates.
}
\end{remark}

\section{ Existence of strong mild solutions }

Here we will use our Theorem \ref{main theorem} to prove existence of strong mild solutions when $B$ grows more than linearly. We will construct such solutions for $\mu$-a.e. initial $x \in H$.

According to Chapter 1 in \cite{O}  (see also \cite{LR}) if $x \in H$ we say that
equation \eqref{uno}  has a (global) {\it strong mild solution} if, for every filtered probability space
$(\Omega,{\cal F}, ({\cal F}_t), \P)$ on which there is defined a cylindrical
$({\cal F}_t)$-Wiener process $W$
there exists an $H$-valued continuous $({\cal F}_t)$-adapted
process $X= (X_t)= (X_t)_{t \ge 0}$ such that
$$(\Omega ,{\cal F}, ({\cal F}_t), \P,W, X)
$$
is a weak mild solution.

\begin{theorem} \label{bt} Let us consider equation \eqref{uno} and assume
 Hypothesis 1 and $B
\in {B_{b, loc}}(H,H)$. Moreover, suppose that
 there
exist $C>0$, $p>0$, such that
\begin{equation}
\left\langle B(  y+z)  ,y\right\rangle \leq C\big(  \left\vert
y\right\vert ^{2}+ e^{p | z |} +1\big),\;\;
 y,\, z \in H.  \label{condition 2}
\end{equation}
Then, for    $\mu$-a.e. $ x \in H$
 (where $\mu = N(0, -\frac{1}{2}A^{-1}))$,
 equation \eqref{uno}  has  a strong mild solution.
 Moreover, this solution is pathwise unique.
\end{theorem}

\begin{remark} \label{ci3} {\em Condition \eqref{condition 2} is a bit stronger than the classical one:
 $\langle B\left(  y\right)  ,y\rangle $ $\leq C\big(  \left\vert
y\right\vert ^{2} +1\big),$  $y \in H$, which is usually imposed in finite dimension to have non-explosion for SDEs with additive noise. We can not use   such condition.
 Indeed for a given mild solution $(X_t) $    we can not write the It\^o
formula for   $|X_t|^2$ due to the fact that our noise is cylindrical.}
\end{remark}

 To prove the result we will use our Theorem \ref{main theorem}
 together with a generalization of the  Yamada-Watanabe theorem (see Theorem 2 in \cite{O} and \cite{LR}) and some a-priori estimates on mild solutions (see Section 4.1).

\begin{example}
{\em  To introduce  an example of a drift $\tilde B$
which satisfies the  assumptions of Theorem \ref{bt}, we first
 consider
a measurable function $g: \R \to \R_+$ such that $g(s) = 0$ if $s \ge 0$ and
 $0 \le g(s) \le C e^{q |s|}$, $s <0$, for some $q>0$. It is easy to check that
 $g(s+ r) r \le C' (1+ |r|^2 + e^{p |s|})$, $s, r \in \R$.
We define $\tilde B : H \to H$,
$$
 \tilde B(x) = \sum_{k \ge 1} \frac{g(x_k)}{k^2} e_k,\;\;\; x \in H.
 $$
It is not difficult to verify that $ \tilde B$ satisfies the assumptions of the previous theorem.  We can also add to our drift $\tilde B$ one of  the singular drifts considered in
Section 4 of \cite{DFPR}; we will still obtain an admissible  drift for our theorem.
}
\end{example}

\subsection{ An a-priori estimate}

 Here we  prove an a-priori estimate
for mild solutions to \eqref{uno} under condition \eqref{condition 2}.
 For this purpose let us consider the OU process
\begin{align*}
Z_t = Z(t,x)= e^{tA} x + \int_0^t e^{(t-s)A} dW_s
\end{align*}
which under our hypotheses has a continuous $H$-valued version. It satisfies
\[
\left\langle Z_{t},\varphi\right\rangle =\int_{0}^{t}\left\langle
Z_{s},A\varphi\right\rangle ds+\left\langle W_{t},\varphi\right\rangle
\]
for all $\varphi\in D\left(  A\right)  $. By Proposition 18 in \cite{DFPR}
we deduce, in particular, that for any $p>0$, $T>0$
\[
K_T := \E\Big[  \sup_{t\in\left[  0,T\right]  }e^{p \left\vert Z_{t}\right\vert
}\Big]  <\infty.
\]
 Recall  that under our hypotheses a weak mild solution to
\eqref{uno} can be defined, equivalently,
 as a tuple  $\left(
\Omega, {\mathcal F},
 ({\mathcal F}_{t}), \P, W, X\right) $, where $\left(
\Omega, {\mathcal F},
 ({\mathcal F}_{t}), \P \right)$ is a  filtered probability space
  on which there is defined a
cylindrical $({\mathcal F}_{t})$-Wiener process $W$ and
 a continuous $({\mathcal F}_{t})$-adapted $H$-valued
process $X = (X_t) = (X_t)_{t \ge 0}$ such that, $\P$-a.s.,
  \[
\left\langle X_{t},\varphi\right\rangle =\left \langle x,\varphi
\right\rangle +\int_{0}^{t}(\left\langle X_{s},A\varphi\right\rangle
+\left\langle B\left(  X_{s}\right)  ,\varphi\right\rangle) ds+\left\langle
W_{t},\varphi\right\rangle,\;\; t \ge 0,
\]
for all $\varphi\in D\left(  A\right)  $ (cf. Chapter 6 of \cite{DZ}).

\begin{theorem} \label{fra2} Assume Hypothesis 1, $B \in B_{b, loc}(H,H) $ and
condition (\ref{condition 2}). Let
 $X= (X_t)_{t \ge 0}$ be a weak mild solution of equation (\ref{uno})
 with $X_0 =x \in H$.

 There exists $ C_p  >0$ (possibly depending on $C$ and $p$ given
 in \eqref{condition 2}) such that
\begin{equation} \label{fra1}
\E\Big[  \sup_{t\in\left[  0,T\right]  }\left\vert X_{t}\right\vert
^{2}\Big]  \leq   e^{C_p T} \left(
|x|^2 +   K_T  +1 \right),\;\; T>0.
\end{equation}
\end{theorem}
\begin{proof}
The process
 $
Y_{t}=X_{t}-Z_{t}%
$
satisfies%
\begin{equation} \label{fra}
\left\langle Y_{t},\varphi\right\rangle =\left\langle x,\varphi
\right\rangle +\int_{0}^{t}(\left\langle Y_{s},A\varphi\right\rangle
+\left\langle B\left(  Y_{s}+Z_{s}\right)  ,\varphi\right\rangle) ds,
\end{equation}
for all $\varphi\in D\left(  A\right)  $, and it has continuous trajectories in
$H$.
 Let us consider \eqref{fra} with $\varphi = e_k$ (see \eqref{e1a}) and
  set $ Y_t^{(k)}= \langle Y_t, e_k\rangle$. Since $
  \frac{d Y_{t}^{(k)}}{dt}
  \cdot  Y_{t}^{(k)}$ $\le  B^{(k)}(Y_t + Z_t) Y_{t}^{(k)} $,
  we
find
$$
\sum_{k \ge 1} \langle Y_t, e_k\rangle ^2 \le  |x|^2 + 2\int_0^t
\left\langle B\left(
Y_{s}+Z_{s}\right)  ,Y_{s}\right\rangle ds.
$$
Hence, by assumption (\ref{condition 2}), for $t \in [0,T]$,
$$
|Y_t|^2  \le |x|^2 +  2 C \int_0^t \left(  \left\vert
Y_s\right\vert ^{2}+ e^{p | Z_s |} +1\right) ds,
$$
and therefore, by the Gronwall lemma,
\begin{align*}
\left\vert Y_{t}\right\vert ^{2}
  \leq e^{CT}\Big(  \left\vert x \right\vert ^{2}+ 2CT (\sup_{s \in \left[
0,T\right] }  e^{p \left\vert
Z_{s}\right\vert } +1) \Big).
\end{align*}
This implies%
\[
\E \Big[  \sup_{t\in\left[  0,T\right]  }\left\vert Y_{t}\right\vert
^{2}\Big]  \leq e^{C_1 T}\Big( \left\vert x\right\vert
^{2}  + \E\Big[  \sup_{s \in \left[
0,T\right] } e^{p \left\vert
Z_{s}\right\vert } \Big]  +1\Big).
\]
Therefore,
\begin{align*}
\E\Big[  \sup_{t\in\left[  0,T\right]  }
\left\vert X_{t}\right\vert
^{2}\Big]    & \leq
2 e^{C_1 T}\Big( \left\vert x\right\vert
^{2}  + \E\Big[  \sup_{s \in \left[
0,T\right] } e^{p \left\vert
Z_{s}\right\vert } \Big]  +1\Big)
  +2 \E\Big[  \sup_{t\in\left[
0,T\right]  }\left\vert Z_{t}\right\vert ^{2}\Big]  \\
& \le   e^{C_p T}\Big( \left\vert x\right\vert
^{2}  + \E\Big[  \sup_{s \in \left[
0,T\right] } e^{p \left\vert
Z_{s}\right\vert } \Big]  +1\Big).
\end{align*}
\end{proof}
\subsection {Proof of Theorem \ref{bt}}

 By Theorem \ref{main theorem} we only  have to prove existence of strong solution for
$\mu$-a.e.  $x \in H$.

We will again consider truncated bounded drifts $B_N = B \, 1_{B(0,N)}$, $N \ge 1$.

By the main result in  \cite{DFPR} there exists
a Borel set $\tilde G \subset H$ with $\mu (\tilde G) =1$ such that for any $x \in
 \tilde G$ we have
pathwise uniqueness for each stochastic equation
\begin{equation} \label{ss}
dX_{t}=(AX_{t}+B_N(X_{t}))dt+dW_{t},\qquad X_{0}=x\in H,
\end{equation}
 $N \ge 1$.
 Let $x \in \tilde G$. By the Girsanov theorem (see Appendix in \cite{DFPR})
 there exists (a unique in law) weak mild solution
 $X_N=(X_N(t))= (X_N(t))_{t\ge 0}$ for each stochastic
 equation \eqref{ss}.

Therefore we can apply  a generalization of the  Yamada-Watanabe theorem
 (see Theorem 2 in \cite{O} and \cite{LR}) to \eqref{ss} when $x \in \tilde G$.

  Let us fix any filtered probability space $\left(
\Omega, {\mathcal F},
 ({\mathcal F}_{t}), \P \right)$
  on which there is defined a
cylindrical $({\mathcal F}_{t})$-Wiener process $W$.
 By the Yamada-Watanabe theorem,  for any $N \ge 1$,
on the fixed
filtered probability space above there exists a (unique) strong mild solution
 $X_N$ to \eqref{ss}. Moreover, since
$$
B_N (x) = B_{N+k}(x),\;\;\; x \in B(0,N),
$$
$k \ge 1$, we have by Theorem \ref{main theorem} that, $\P$-a.s.,
$$
\tau_N:=\tau_N^{X_N} = \tau_N^{X_{N+k}}, \;\; k \ge 1, \; N \ge 1,
$$
and $X_N(t \wedge \tau_N)= X_{N + k}(t \wedge \tau_N)$, $t \ge 0$.

 It is enough to  construct the
 strong solution $X$ to \eqref{uno} on $ [0,T]$ for a fixed $T>0$.
We define an $H$-valued  stochastic process $X $ on  $ \Omega' =\cup_{N \ge 1}
\{ \tau_N >T \}$  as
$$
X(t)(\omega) := X_N(t)(\omega),\;\;\;  t \in [0,T],
$$
if $\omega \in \{ \tau_N >T \}$ (we set $X_t(\omega) =0$
  if $\omega \not \in \Omega'$, $t \in [0,T]$). Then $X(t)$ is well defined.

It is not difficult to  prove that $X$ is a strong mild solution
 on $[0,T]$ if we show that
\begin{equation}\label{t4}
    \lim_{N \to \infty}\P (\tau_N >T ) =1
\end{equation}
(this will imply that $\P (\Omega' ) =1$).
To verify \eqref{t4} we will apply  Theorem \ref{fra2}. Note that  each $B_N$  satisfies
$$
\left\langle B_N\left(  y+z\right)  ,y\right\rangle \leq C\big(  \left\vert
y\right\vert ^{2}+ e^{p | z |} +1\big),\;\;
 y,\, z \in H,
$$
with the same constants $C$ and $p$ of \eqref{condition 2}.
 By Theorem \ref{fra2} we obtain
\begin{equation}\label{r555}
\E \Big[  \sup_{t\in\left[  0,T\right]  }\left\vert X_N (t)\right\vert
^{2}\Big]  \leq   e^{C_p T} \left(
|x|^2 +  K_T  +1 \right),
\end{equation}
with $C_p$ independent of $N \ge 1.$  Since
$$
\P (\tau_N  \le T) = \P \big(\sup_{t\in\left[  0,T\right]  }
\left\vert X_N (t)\right\vert
 \ge N \big)
$$
by \eqref{r555} and  the Chebychev inequality  we  easily get
  assertion \eqref{t4} from \eqref{r555}
 and this completes the proof.

\subsection{Existence and uniqueness of local mild solutions}

Finally, let us discuss a possible  extension of our result
to  the case when the drift term $B$  only belongs to $B_{b,loc}\left(
H,H\right)  $,  without  requiring  hypothesis \eqref{condition 2}.

\smallskip
We need the concept of local solution (see, for instance, \cite{ABW} for some
additional facts about local solutions).
Let $\left(  \Omega, \mathcal{F},(\mathcal{F}_{t}),\P\right)  $ be a filtered
probability space. A stopping time $\tau:\Omega\rightarrow\left[
0, +\infty\right]  $ is called \textit{accessible} if there exists a sequence of
stopping times $(\tau_n) =\left(  \tau_{n}\right)  _{n\in\mathbb{N}}$ such that
$\P\left(  \tau_{n}<\tau\right)  =1$ and $\P\left(  \lim_{n\rightarrow\infty
}\tau_{n}=\tau\right)  =1$. The previous sequence $\left(  \tau_{n}\right)$ is called an approximating sequence of $\tau.$

Notice that, if $\tau_{1}$ and $\tau_{2}$ are
accessible stopping times, then also $\tau=\tau_{1}\wedge\tau_{2}$ is an
accessible stopping time.

\smallskip Let $\tau$ be an accessible stopping time and consider
 $[0, \tau) \times \Omega $ $= \{ (t, \omega) \in [0, +\infty) \times \Omega :$ $
  0 \le t < \tau(\omega) \}$.
 An $H$-valued  stochastic process $X$ defined  on
$[0,\tau  )$
  (i.e., $X : [0, \tau) \times \Omega  \to H$)
is called $(\mathcal{F}_t)$-adapted if $X_{t}(\cdot):$
$\{ t < \tau \} \to H$ is $\mathcal{F}_t$-measurable, for any $t \ge 0$
 (on $\{ t < \tau \}$ we consider the restricted $\sigma$-algebra
 $\{ A \cap \{ t < \tau \}\}_{A \in {\cal F}_t}$); moreover, it is called continuous if  trajectories are continuous on
$[0,\tau)$, $\P$-a.s..  Note that
$X$ is $(\mathcal{F}_t)$-adapted if and only if the process $\tilde X =(\tilde X_t)_{t\ge 0}$,
\begin{equation} \label{f78}
 \tilde X_t  = X_t 1_{ \{ \tau > t \} } + 0 \cdot 1_{ \{ \tau \le t \} }
\end{equation}
 is $(\mathcal{F}_t)$-adapted.

\begin{definition} \label{www} Let $x \in H$.
 We call  { \it local weak mild solution}  to
(\ref{uno}) a tuple   $(
\Omega, {\mathcal F},
 ({\mathcal F}_{t}),$ $ \P, W, X, \tau) $, where $\left(
\Omega, {\mathcal F},
 ({\mathcal F}_{t}), \P \right)$ is a  filtered probability space
  on which it is defined a
 cylindrical $({\mathcal F}_{t})$-Wiener process $W$,
  an accessible stopping time $\tau$  and
 a continuous $({\mathcal F}_{t})$-adapted $H$-valued
process $X $ defined on $[0, \tau)$
 such that, there exists an approximating sequence $(\tau_n)$ of $\tau$ for which, $\P$-a.s.,
  on $\{ t \le \tau_n \}$
\begin{equation*}
X_{t}=e^{tA}x+\int_{0}^{t}e^{\left(  t-s\right)  A}B\left(
X_{s}\right)
ds+\int_{0}^{t}e^{\left(  t-s\right)  A}dW_{s},
\end{equation*}
for all $n\in\mathbb{N}$ and $t\geq0$.

A local weak mild solution $X$ which is $(\bar {\cal F}_t^W)$-adapted (here $(\bar{\cal F}_t^W)$ denotes the  completed  natural filtration of the cylindrical process $W$) and such that $\tau$ is an $(\bar {\cal F}_t^W)$-stopping time
is called \textit{a local strong mild solution}.
\end{definition}

\begin{theorem} Assuming  Hypothesis 1 and $B \in B_{b,loc}\left(H,H\right)  $,
existence  of \textit{local} strong mild solutions holds for $\mu$-a.e. initial condition $x \in H$. Moreover, for $\mu$-a.e.  $x \in H$, if $X$ and $Y$ are two local weak mild solutions  on the
same   $\left(  \Omega,\mathcal{F},(\mathcal{F}
_{t}),\P\right)  $ with the same cylindrical Wiener process $W$, defined on   $[0,\tau^{X})$ and $[0,\tau^{Y})$
respectively ($\tau^X$ and $\tau^Y$ are accessible stopping times
 as in Definition
 \ref{www}), then,  $\P$-a.s., $X=Y$ on $[0, \tau)$,
 where $\tau=\tau^{X}\wedge\tau^{Y}$.
\end{theorem}
\begin{proof} Let us sketch some of the details of the proof.

\smallskip
From the first part of the proof of Theorem 16
 (see also Theorem 1 in \cite{DFPR}), given a-priori a filtered
probability space $\left(  \Omega,\mathcal{F},(\mathcal{F}_{t}),\P\right)  $ and
a  cylindrical $({\cal F}_t)-$Wiener process
$W$, we have the existence of a local strong mild solution $X_N$
on $[0,\tau_{N})$, for every $N\in\mathbb{N}$.  Note that each
 $\tau_{N}$ is accessible since as approximating sequence we may take
 $\tau_{N,n} = \inf \{t \ge 0 \; :\; X_N(t) \not \in B(0,N - \frac{1}{n})\} \wedge n$,
 for $n \ge 1$.

 Thus, taking $\tau
=\sup_{N\in\mathbb{N}}\tau_{N}$,  $\tau$ is accessible and we have
a local strong mild solution on
$[0,\tau)$.

\smallskip In order to prove uniqueness, we first note that if $X$ is a  local weak mild solution defined on  $[0,\tau^{X})$ and  $(\tau_n^X)$ is an approximating sequence of $\tau^X$ as in the definition of solution, then assertion \eqref{e22} of Theorem \ref{p4} holds, for any $t>0$,  $n \ge 1,$ $\mathbb \P\mbox{\rm
-a.s.}$, on the event  $\{ t \le \tau_{n,N}^X \}$, where
$\tau_{n,N}^X$ is the stopping time
$$
\tau_{n,N}^X = \inf \{t \in [0, \tau^X) \; :\; X_t \not \in B(0,N)\} \wedge
\tau_n^X
$$
 ($\tau_{n,N}^{X}= \tau_n^X$ if the set is empty; to show that $\tau_{n,N}^X$ is a stopping time, note that $\tau_{n,N}^X = \inf \{t \ge 0 \; :\; \tilde X_t \not \in B(0,N)\}$ $\wedge
\, \tau_n^X$, where $\tilde X_t$ is defined in \eqref{f78} with $\tau$ replaced by
 $\tau^X$).

Indeed, one can repeat the arguments in the proof of Theorem \ref{p4} with the same functions $u_N$,  replacing  $\tau_N^X$ with $\tau_{n,N}^{X}
$.

\smallskip
 Now let  $X$ and $Y$ be two local weak mild solutions  as in the second part of the theorem.

If $(\tau_n^X)$ and $(\tau_n^Y)$ are, respectively, approximating sequences of  $\tau^X$ and $\tau^Y$ as in the definition of solution, then in order to prove  uniqueness, it is enough to consider $\sigma_n = \tau^{X}_n\wedge\tau^{Y}_n$ and check that, $\P$-a.s.,
\begin{equation}
X =Y \;\; \text{on} \;\; [0, \sigma_n],\;\; n \ge 1.
\end{equation}
Let us fix $n \ge 1$. We can adapt  the proof of Theorem \ref{main theorem}, arguing on the interval
$[0, \eta \wedge T \wedge \tau_{n,N}^X \wedge \tau_{n,N}^Y]$ (cf. \eqref{r88}).
  We finally get that
 $X =Y $ on $[0, T \wedge \sigma_n]$, $T>0$, and this gives the assertion.
\end{proof}

\vskip 4mm
\noindent
\textbf{Acknowledgement.}
The authors would like to thank the  referees
 for their
 useful remarks and suggestions.

\end{document}